\documentclass[11pt,reqno,a4paper]{amsart}
\usepackage[english]{babel}
\usepackage{enumitem,xcolor}
\setlist[enumerate,1]{label={\textup{(\arabic*)}}}
\setlist[enumerate,2]{label={\textup{(\roman*)}}}
\usepackage{amsmath,amsthm,amssymb}
\usepackage[all]{xy}
\SelectTips{cm}{11}
\usepackage[margin=3cm]{geometry}

\theoremstyle{definition}
\newtheorem{definition}{Definition}[section]
\theoremstyle{plain}
\newtheorem{theorem}[definition]{Theorem}
\newtheorem{proposition}[definition]{Proposition}
\newtheorem{corollary}[definition]{Corollary}
\newtheorem{lemma}[definition]{Lemma}

\theoremstyle{remark}
\newtheorem{remark}[definition]{Remark}
\newtheorem{example}[definition]{Example}

\makeatletter
\newlength{\@thlabel@width}%
\newcommand{\thmenumhspace}{\settowidth{\@thlabel@width}{\upshape(i)}\sbox{\@labels}{\unhbox\@labels\hspace{\dimexpr-\leftmargin+\labelsep+\@thlabel@width-\itemindent}}}
\makeatother

\title[Generalized Gottlieb and Whitehead center groups of space forms]{Generalized Gottlieb and Whitehead center groups of space forms}

\author{Marek Golasi\'nski}
\address{Faculty of Mathematics and Computer Science\\
University of Warmia and Mazury\\
S\l oneczna 54 Street\\
10-710 Olsztyn, Poland}
\email{marekg@matman.uwm.edu.pl}

\author[Thiago de Melo]{Thiago de Melo}
\address{Instituto de Geoci\^encias e Ci\^encias Exatas\\ UNESP--Univ Estadual Paulista\\ Av.~24A, 1515, Bela Vista. CEP 13.506--900. Rio Claro--SP, Brazil}
\email{tmelo@rc.unesp.br}

\thanks{Both authors are supported by CAPES--Ci\^encia sem Fronteiras. Processo: 88881.068125/2014-01.}

\begin{document}\baselineskip=1.5em

\begin{abstract}We extend the Oprea's result $G_1(\mathbb{S}^{2n+1}/H)=\mathcal{Z}H$ to the  $1$\textsuperscript{st} generalized Gottlieb group $G_1^f(\mathbb{S}^{2n+1}/H)$ for a map $f\colon A\to \mathbb{S}^{2n+1}/H$. Then, we compute or estimate the groups $G_m^f(\mathbb{S}^{2n+1}/H)$ and $P_m^f(\mathbb{S}^{2n+1}/H)$ for some $m>1$ and finite groups $H$. 
\end{abstract}

\maketitle

\section*{Introduction}
Throughout this paper, all spaces are path-connected with homotopy types of $CW$-complexes, maps and homotopies are based. We use the standard terminology
and notations from the homotopy theory, mainly from \cite{juno-marek} and \cite{toda}. We do not distinguish
between a map and its homotopy class and we neglect the base-point in notation in case it is needed.

Let $X$ be a connected space and $\mathbb{S}^m$ the $m$-sphere. The  \textit{$m$\textsuperscript{th} Gottlieb group} $G_m(X)$ of $X$, defined first for $m=1$ in \cite{gottlieb1} and then for $m\ge 1$ in \cite{gottlieb}, is the subgroup of
the $m$\textsuperscript{th} homotopy group $\pi_m(X)$ consisting of all elements which can be
represented by a map $\alpha \colon \mathbb{S}^m\to X$ such that $\mathrm{id}_X \vee \alpha \colon X\vee \mathbb{S}^m
\to X$ extends (up to homotopy) to a map $F \colon X\times \mathbb{S}^m\to X$.
Following \cite{gottlieb}, we recall that $P_m(X)$ is the set of elements  $\alpha\in \pi_m(X)$ whose Whitehead products $[\alpha,\beta]$ are zero for all $\beta\in \pi_l(X)$ with $l\geq 1$. It turns out that $P_m(X)$
forms a subgroup of $\pi_m(X)$ called the \textit{Whitehead center group} and, by \cite[Proposition~2.3]{gottlieb}, it holds $G_m(X)\subseteq P_m(X)$. 
Some advanced attempts to compute the groups $G_m(X)$ and $P_m(X)$ for spheres and projective spaces have been made in \cite{juno-marek}.

Now, given a map $f \colon A\to X$, in view of \cite{gottlieb} (see also \cite{kim}),
the  \textit{$m$\textsuperscript{th} generalized Gottlieb group} $G^f_m(X)$ is defined 
as the subgroup of the $m$\textsuperscript{th} homotopy group $\pi_m(X)$ consisting of all elements which can be
represented by a map $\alpha \colon \mathbb{S}^m\to X$ such that $f\vee \alpha \colon A\vee \mathbb{S}^m\to X$ extends (up to homotopy) to a map $F \colon A\times \mathbb{S}^m\to X$. If $A=X$ then the group $G_1^f(X)$, also denoted by $J(f)$, is called the \textit{Jiang group} of the map $f\colon X \to X$ in honor of Bo-Ju Jiang  who recognized in \cite{jiang} their importance to the Nielsen--Wecken theory of fixed point classes. The role the group $J(f)$ played in that theory has been intensively studied in the book \cite{brownr} as well.

The  \textit{$m$\textsuperscript{th} generalized Whitehead center group}  $P^f_m(X)$, as defined in \cite{kim}, is  the set of all elements $\alpha\in\pi_m(X)$ whose Whitehead products $[\alpha,f\beta]$ are zero for all  $\beta\in\pi_l(A)$ with $l\ge 1$. It turns out that $P^f_m(X)$ forms a subgroup of $\pi_m(X)$ and $G^f_m(X)\subseteq P^f_m(X)$. 

Given a free action  of a finite group $H$ on $\mathbb{S}^{2n+1}$, Oprea \cite{oprea} has shown that $G_1(\mathbb{S}^{2n+1}/H)=\mathcal{Z}H$, the center of $H$.  Further, in the special case of a linear action of $H$ on $\mathbb{S}^{2n+1}$, a very nice representation-theoretic proof of that fact has been given in \cite{bro}. The first main advantage that we take in this paper is to show that $G_1^f(\Sigma({2n+1})/H)=\mathcal{Z}_Hf_*(\pi_1(A))$, the centralizer of $f_\ast(\pi_1(A))$ in $H$ for a map $f\colon A\to \Sigma(2n+1)/H$, where $\Sigma(2n+1)$ is a $(2n+1)$-homotopy sphere. In particular, we get that the Jiang group  $J(f)=\mathcal{Z}_Hf_*(H)$ for any $f\colon \Sigma(2n+1)/H\to \Sigma(2n+1)/H$. Further, we compute or estimate Gottlieb groups $G_m^f(\Sigma(2n+1)/H)$ and $P_m^f(\Sigma(2n+1)/H)$ for some $m>1$,  finite groups $H$ and $f\colon A \to \Sigma(2n+1)/H$. In particular, we get some results on  $G_m^f(L^{2n+1}(l; q_1,\dots,q_n))$ for the $(2n+1)$-dimensional generalized lens space $L^{2n+1}(l; q_1,\dots,q_n)$ and a map $f\colon A \to L^{2n+1}(l; q_1,\dots,q_n)$.

In Section~\ref{sec.1}, Proposition~\ref{prop.main} generalizes the Gottlieb's result \cite{gottlieb1} and states that\linebreak $G_1^f(K(\pi,1))=\mathcal{Z}_\pi f_\ast(\pi_1(A))$ for $f\colon A\to K(\pi,1)$
and Theorem~\ref{PP} states: 

\textit{If $p \colon  \tilde{X}\to X$  is a  covering map of a space $X$ and  $f \colon  A\to \tilde{X}$ then the isomorphism $p_\ast \colon  \pi_m(\tilde{X})\to \pi_m(X)$ for $m> 1$
restricts to isomorphisms \[p_{\ast|} \colon  G_m^f(\tilde{X})\to  G_m^{pf}(X) \text{ and }p_{\ast|} \colon  P_m^f(\tilde{X})\to  P_m^{pf}(X)\] for $m>1$.}

The main result of that section, generalizing \cite[\textsc{Theorem~A}]{oprea}, is Theorem~\ref{main} which implies that \[G_1^f(\Sigma(2n+1)/H)=\mathcal{Z}_H f_*(K)\] for $f\colon \Sigma(2d+1)/K \to \Sigma(2n+1)/H$ with $d\leq n$.
Then, in Corollary~\ref{CCC}, we derive:
\begin{enumerate}
\item $P_1^{\gamma_{2n+1}}(\Sigma(2n+1)/H)=G_1^{\gamma_{2n+1}}(\Sigma(2n+1)/H)=H$;

\item $P_m^{\gamma_{2n+1}}(\Sigma(2n+1)/H)=\gamma_{{2n+1}\ast}(P_m(\Sigma(2n+1)))=\gamma_{{2n+1}\ast}(G_m(\mathbb{S}^{2n+1}))$
for $m>1$, where $\gamma_{2n+1} \colon \Sigma(2n+1)\to  \Sigma(2n+1)/H$ is the quotient map.
\end{enumerate}

Section~\ref{sec.2} makes use of some results from \cite{juno-marek} to take up the systematic study of the groups $G_m(\Sigma(2n+1)/\mathbb{Z}_l)$ for some $m>1$. 

Section~\ref{sec.3} applies \cite{juno-marek} to present computations of $G_m(\mathbb{S}^{2n+1}/H)$ for $m>1$ and $H<\mathbb{S}^3$.

Finally, Section~\ref{sec.4} concludes with some $G^f_m(L^{2n+1}(l;\mathbf{q}))$ for the lens space 
$L^{2n+1}(l;\mathbf{q})=L^{2n+1}(l; q_1,\dots,q_n)$. Further, for the quotient map $\gamma_{2n+1} \colon \mathbb{S}^{2n+1}\to L^{2n+1}(l;\mathbf{q})$, Proposition~\ref{lens} states:
\begin{enumerate}
\item

$P_1^{\gamma_{2n+1}}(L^{2n+1}(l;\mathbf{q}))=P_1(L^{2n+1}(l;\mathbf{q}))=G_1^{\gamma_{2n+1}}(L^{2n+1}(l;\mathbf{q}))=G_1(L^{2n+1}(l; \mathbf{q}))=\mathbb{Z}_l$;
\item

$P_m^{\gamma_{2n+1}}(L^{2n+1}(l;\mathbf{q}))=P_m(L^{2n+1}(l;\mathbf{q}))=G_m^{\gamma_{2n+1}}(L^{2n+1}(l;\mathbf{q}))=\gamma_{{2n+1}\ast}(G_m(\mathbb{S}^{2n+1}))$ for $m>1$.
\end{enumerate}

\section{Generalized Gottlieb and Whitehead center groups}\label{sec.1}

Given spaces $A$ and $X$, write $X^A$ for the space of continuous maps from $A$ into $X$ with the compact-open topology.
Next, consider the evaluation map $\mathrm{ev} \colon X^A\to X$, i.e., $\mathrm{ev}(f)=f(a_0)=x_0$ for $f\in  X^A$ and the base-point $a_0\in A$. 
Then, it holds \[G^f_m(X)=\operatorname{Im}(\mathrm{ev}_\ast \colon \pi_m(X^A,f)\to \pi_m(X,x_0))\] for $m\ge 1$.
Notice that $G^f_m(X)=P^f_m(X)=\pi_m(X)$ for any $m\geq 1$ provided $f\in P_n(X)$ for some $n\ge 1$.

Further, it holds:
\begin{remark}\label{rem.homoteq}  Let  $f\colon A\to X$.
\begin{enumerate}

\item \label{rem.homoteq.3} Any map $g \colon A'\to A$ leads to the inclusion relations $G^f_m(X)\subseteq G^{fg}_m(X)$ and $P^f_m(X)\subseteq P^{fg}_m(X)$ for $m\ge 1$.

\item\label{rem.homoteq.1} If $g\colon A'\to A$ is a homotopy equivalence then  $G^f_m(X)= G^{fg}_m(X)$ and $P^f_m(X)= P^{fg}_m(X)$ for $m\geq 1$.

\item\label{rem.homoteq.2} If  $f_\ast \colon  \pi_1(A)\to \pi_1(X)$ is the induced homomorphism by $f\colon A \to X$ and $\pi_1(X)$ acts trivially on $\pi_m(X)$ for all $m>1$ then $P_1^f(X)=\mathcal{Z}_{\pi_1(X)}f_\ast(\pi_1(A))$, the centralizer of the image $f_\ast(\pi_1(A))$ in $\pi_1(X)$. If  $f_\ast \colon  \pi_\ast(A)\to \pi_\ast(X)$ is an epimorphism then $P_m^f(X)=P_m(X)$ for all $m\ge 1$.

\item \label{rem.homoteq.4} Given a map $h \colon X\to Y$ the induced homomorphism $h_\ast \colon \pi_m(X)\to \pi_m(Y)$ restricts to homomorphisms ${h_{\ast|}} \colon G^f_m(X)\to G^{hf}_m(Y)$ and ${h_{\ast|}} \colon P^f_m(X)\to P^{hf}_m(Y)$.

In particular, for $f=\mathrm{id}_X$, we get homomorphisms $h_{\ast|}\colon G_m(X)\to G^h_m(Y)$ and $h_{*|}\colon P_m(X)\to P^h_m(Y)$.

\item\label{rem.homoteq.5} If $A=\mathbb{S}^k$ then $G_m^f(X)=P_m^f(X)=\operatorname{Ker}[f,-]$ for $m\ge 1$.

\item\label{rem.homoteq.6} If $\alpha\colon \mathbb{S}^l\to \mathbb{S}^m$ then the induced map $\alpha^\ast \colon \pi_m(X)\to \pi_l(X)$ restricts to maps
$\alpha_|^\ast \colon G_m^f(X)\to G_l^f(X)$ and $\alpha_|^\ast \colon P_m^f(X)\to P_l^f(X)$.

If the map $\alpha=E\beta$,
the suspension of $\beta$, then:
\begin{enumerate}
\item\label{rem.homoteq.6.1}  those restricted maps $\alpha^*_|$ are homomorphisms;
\item\label{rem.homoteq.6.2} $(E^{m-n}\alpha)_|^\ast\colon G_m ^\alpha(\mathbb{S}^n)\to G_{m-n+l}(\mathbb{S}^n)$ for $m\ge n$.
\end{enumerate}
\end{enumerate}
\end{remark}

\begin{proof} Because \ref{rem.homoteq.3}--\ref{rem.homoteq.6}\ref{rem.homoteq.6.1} are obvious, we show only \ref{rem.homoteq.6}\ref{rem.homoteq.6.2} applying the following property of the Whitehead product: \[[\gamma E\delta,\gamma' E\delta']=[\gamma,\gamma']E(\delta\wedge \delta')\]
for $\gamma\in\pi_s(X)$, $\gamma'\in\pi_t(X)$ and maps $\delta\colon \mathbb{S}^{r-1}\to\mathbb{S}^{s-1}$, and 
$\delta' \colon \mathbb{S}^{r'-1}\to\mathbb{S}^{t-1}$.

Now, notice that we may assume $l,m\ge n$ and write $\iota_k$ for the identity map of the sphere $\mathbb{S}^k$. Then, given $\nu\in G_m^\alpha(\mathbb{S}^n)$, we get:
\begin{align*}
0=[\nu,\alpha]&=[\nu E\iota_{m-1},\iota_n E\beta]=[\nu,\iota_n]E(\iota_{m-1}\wedge\beta)\\ &=[\nu,\iota_n]E^m\beta=[\nu,\iota_n]E(E^{m-n}\beta\wedge\iota_{n-1})=[\nu E^{m-n+1}\beta,\iota_n]=[\nu E^{m-n}\alpha,\iota_n].
\end{align*}

Consequently, \[(E^{m-n}\alpha)_|^\ast\colon G_m ^\alpha(\mathbb{S}^n)\to G_{m-n+l}(\mathbb{S}^n)\] and the proof is complete.
\end{proof} 

\begin{example} \begin{enumerate}\thmenumhspace
\item Let $f\colon \mathbb{S}^m\to \mathbb{S}^m$ and write $\deg f$ for its degree. Then, $G_m^f(\mathbb{S}^m)=P_m^f(\mathbb{S}^m)=(\deg f)G_m(\mathbb{S}^m)$ provided $\deg f\neq 0$ and $G_m^f(\mathbb{S}^m)=P_m^f(\mathbb{S}^m)=\pi_m(\mathbb{S}^m)$ otherwise. Further, recall that $G_n(\mathbb{S}^n)=0$ for $n$ even and 
\[G_n(\mathbb{S}^n)=P_n(\mathbb{S}^n)=\begin{cases}\pi_n(\mathbb{S}^n),&\text{for $n=1,3,7$};\\
2\pi_n(\mathbb{S}^n), &\text{for $n\neq 1,3,7$ odd.}
\end{cases}\]
Hence, $G_n^f(\mathbb{S}^n)=0$ for even $n$ and by the Whitehead product $2[\iota_n,\iota_n]=0$ for odd $n$ , we derive that  \[G^f_n(\mathbb{S}^n)=P_n^f(\mathbb{S}^n)=\begin{cases}\pi_n(\mathbb{S}^n),& \text{for $\deg f$ even};\\
G_n(\mathbb{S}^n),& \text{for $\deg f$ odd,}
\end{cases}\]
provided $n$ is odd.

\item Consider a finite group $H$ with a  free action $H\times\mathbb{S}^n\to \mathbb{S}^n$ for $n\ge 1$, write $\mathbb{S}^n/H$ for the associated orbit space
and $\gamma_n \colon \mathbb{S}^n\to \mathbb{S}^n/H$ for the quotient map.

Let $f\colon \mathbb{S}^k\to\mathbb{S}^n/H$ with $k\ge 1$. If $k>1$ then $f=\gamma_n f'$ for some $f' \colon \mathbb{S}^k\to \mathbb{S}^n$
and \[G_m^f(\mathbb{S}^n/H)=P_m^f(\mathbb{S}^n/H)=\begin{cases}\gamma_{n\ast} G_m^{f'}(\mathbb{S}^n), &\text{if $k>1$;}\\
\gamma_{n\ast}\pi_m(\mathbb{S}^n),& \text{if $k=1$,}
\end{cases}\] for $m>1$ and
\[G_1^f(\mathbb{S}^n/H)=P_1^f(\mathbb{S}^n/H)=\begin{cases} H, &\text{if $k>1$;}\\
\mathcal{Z}_H\big<f\big>,& \text{if $k=1$.}
\end{cases}\]
\end{enumerate}
\end{example}

Now, we show:
\begin{lemma} If $f \colon  A\to X$  then $P_1^f(X)\subseteq\mathcal{Z}_{\pi_1(X)}f_\ast(\pi_1(A))$.
\end{lemma}
\begin{proof}
Take $\alpha\in P_1^f(X)$. Then, the Whitehead product $[\alpha,f\beta]=0$ for all $\beta\in\pi_1(A)$.
This implies $\alpha f_\ast(\beta)=f_\ast(\beta)\alpha$ and the proof follows.
\end{proof}

Because  $G_1^f(X)\subseteq P_1^f(X)$, we derive that \[G_1^f(X)\subseteq\mathcal{Z}_{\pi_1(X)}f_\ast(\pi_1(A))\] 
which for $f=\mathrm{id}_X$ implies the result of Gottlieb \cite{gottlieb1}.

Gottlieb \cite[Corollary~I.13]{gottlieb1} has shown that $G_1(K(\pi,1))=\mathcal{Z}\pi$, the center of the group $\pi$. We generalize that result as follows: 

\begin{proposition}\label{prop.main} If   $f \colon  A\to K(\pi,1)$ then \(G_1^f(K(\pi,1))=\mathcal{Z}_\pi f_\ast(\pi_1(A))\).
\end{proposition}
\begin{proof}
By the above we have $G_1^f(K(\pi,1))\subseteq\mathcal{Z}_\pi f_\ast(\pi_1(A))$.

To show the opposite inclusion, take $\alpha\in \mathcal{Z}_\pi f_\ast(\pi_1(A))$ and 
consider the homomorphism $\varphi \colon  \pi_1(A\times \mathbb{S}^1)=\pi_1(A)\times\mathbb{Z}\to \pi$ given by $\varphi(g,n)=\alpha^nf_\ast(g)$ for
$(g,n)\in  \pi_1(A)\times\mathbb{Z}$, where $\mathbb{Z}$ is the group of integers.
Then, applying the result \cite[(4.3)~Theorem]{wh}, we get the required map $A\times \mathbb{S}^1\to K(\pi,1)$.

Nevertheless, we decide to sketch a direct proof. The map $f\vee \alpha \colon  A\vee \mathbb{S}^1\to K(\pi,1)$ and the inclusion map $i \colon   A\vee \mathbb{S}^1\hookrightarrow A\times \mathbb{S}^1$
induce homomorphisms $(f\vee \alpha)_\ast \colon  \pi_1(A)\ast\mathbb{Z}\to \pi$ and $i_\ast \colon   \pi_1(A)\ast\mathbb{Z}\to \pi_1(A)\times \mathbb{Z}$ with $\varphi i_\ast=(f\vee \alpha)_\ast$. 

Hence, by \cite[Remark, p.\ 846]{gottlieb1} the map
$f\vee \alpha \colon  A\vee \mathbb{S}^1\to K(\pi,1)$ is $2$-extensible. Because $\pi_m(K(\pi,1))=0$ for $m>1$, we derive an extension
$A\times \mathbb{S}^1\to K(\pi,1)$ of the map $f\vee \alpha $ and the proof is complete.
\end{proof}

We point out that the inclusion $\mathcal{Z}_\pi f_\ast(\pi)\subseteq G_1^f(K(\pi,1))$ for any self-map  $f\colon K(\pi,1) \to K(\pi,1)$  was already obtained in \cite[Chapter~VII, Theorem~10]{brownr}.

\begin{example}\begin{enumerate}\thmenumhspace
\item If $\pi$ is an abelian group then $G_1(K(\pi,1))=G_1^f(K(\pi,1))=\pi$ for any $f\colon A\to K(\pi,1)$.

\item Let $Q_8=\langle i,j \rangle $ be the quaternionic group. Because the center $\mathcal{Z}(Q_8)=\mathbb{Z}_2$ we derive that $G_1(K(Q_8,1))=\mathbb{Z}_2$. Let $f\colon \mathbb{S}^1\to K(Q_8,1)$ be the map determined by the homomorphism $\mathbb{Z}\to Q_8$ such that $1\mapsto i$. Then, $G_1^f(K(Q_8,1))=\mathcal{Z}_{Q_8}\langle i \rangle=\langle i \rangle$ and so we have the proper inclusion $G_1(K(Q_8,1))\subsetneqq G_1^f(K(Q_8,1))$.
\end{enumerate}
\end{example}

Recall that a space $X$ is said to be \textit{aspherical} if $\pi_m(X,x)=0$ for $m>1$ and all $x\in X$. Following
the ideas stated in \cite[Section~III, $X^X$]{gottlieb1}, we can easily generalize \cite[Theorem~III.2]{gottlieb1}
and \cite[Theorems~6.1 and 6.2]{hu} as follows: 
\begin{proposition}\label{aspherical} If $X$ is a locally finite, aspherical, pathwise connected space and $f \colon A\to X$ then:
\begin{enumerate}
\item $\pi_1(X^A,f)=G_1^f(X)=\mathcal{Z}_{\pi_1(X)}f_\ast(\pi_1(A))$;

\item $\pi_m(X^A,f)=0$ for  $m>1$.
\end{enumerate}
\end{proposition}
In particular, if $A$ is a $1$-connected space then $X^A$ is  pathwise connected and  $\pi_1(X^A,f)=G_1^f(X)=\mathcal{Z}\pi_1(X)$ for any $f \colon A\to X$.

\medskip

Further, Gottlieb \cite[Theorems~6-1 and  6-2]{gottlieb} has shown:
\begin{proposition}\label{l1} If $p \colon  \tilde{X}\to X$ is a covering map then
\[ p_*^{-1}(G_m(X))\subseteq G_m(\tilde{X}) \text{ for $m\geq 1$.}\] 
\end{proposition}

For Whitehead center groups, we get:
\begin{proposition}\label{L} If $p \colon  \tilde{X}\to X$ is a covering map then
\[ p_*^{-1}(P_m(X))\subseteq P_m(\tilde{X}) \text{ for $m\geq 1$.}\] Further, if $X$ is a simple space then \[ p_*^{-1}(P_m(X))= P_m(\tilde{X}) \text{ for $m\geq 1$.}\]
\end{proposition}

\begin{proof}
Since $p_*\colon\pi_m(\tilde{X})\to \pi_m(X)$ is a monomorphism, the inclusion $p_*^{-1}(P_m(X))\subseteq P_m(\tilde{X})$ for $m\geq 1$ is straightforward. 

Let now $X$ be a simple space and take $\alpha\in P_m(\tilde{X})$, and $\beta\in \pi_k(X)$. If $k=1$ then $[p\alpha,\beta]=0$ since $X$ is a simple space. If $k>1$ then there is $\gamma\in \pi_k(\tilde{X})$ such that $p\gamma=\beta$. Hence, $[p\alpha,\beta]=[p\alpha,p\gamma]=p_*[\alpha,\gamma]=0$ and the proof follows.
\end{proof}

To state next result, we prove:
\begin{lemma}\label{lemma7}If $p\colon \tilde{X}\to X$ is a covering map and $f\colon A\to \tilde{X}$ then:
\begin{enumerate}
\item $G_1^f(\tilde{X})= p_*^{-1}(G_1^{pf}(X))$;
\item $P_1^f(\tilde{X})= p_*^{-1}(P_1^{pf}(X))$.
\end{enumerate}
\end{lemma}

\begin{proof}Because $p_*\colon \pi_1(\tilde{X})\to \pi_1(X)$ restricts to $p_{*|}\colon G_1^f(\tilde{X})\to G_1^{pf}(X)$ and  $p_{*|}\colon P_1^f(\tilde{X})\to P_1^{pf}(X)$, we deduce that $G_1^f(\tilde{X})\subseteq p_*^{-1}(G_1^{pf}(X))$ and $P_1^f(\tilde{X})\subseteq p_*^{-1}(P_1^{pf}(X))$.

First, we show $p_*^{-1}(G_1^{pf}(X))\subseteq G_1^f(\tilde{X})$. Given $\alpha\in p_*^{-1}(G_1^{pf}(X)) $, we get a map $F\colon A\times \mathbb{S}^1\to X$ such that the diagram 
\[ \xymatrix@C=2cm{A\vee \mathbb{S}^1\ar@{^{(}->}[d] \ar[r]^-{pf\vee p\alpha} & X \\ 
A\times \mathbb{S}^1 \ar[ru]_F &} \] commutes
up to homotopy. But, for any $\beta\colon \mathbb{S}^1\to A\times \mathbb{S}^1$ there is a map $\beta'\colon \mathbb{S}^1 \to A\vee \mathbb{S}^1$ homotopic to $\beta$. Then, $F\beta= F\beta' = p(f\vee \alpha)\beta' $ and  consequently $F_*(\pi_1(A\times \mathbb{S}^1))\subseteq p_*(\pi_1(\tilde{X}))$. Hence, there is a lifting  $\tilde{F}\colon A\times \mathbb{S}^1 \to \tilde{X}$  of the map $F$. Because $p\tilde{F}_{|A\vee\mathbb{S}^1}=F_{|A\vee\mathbb{S}^1}=p(f\vee\alpha)$, we derive that $\tilde{F}_{|A\vee\mathbb{S}^1}=f\vee\alpha$ and so $\alpha\in G_1^f(\tilde{X})$.

Now, we show $p_*^{-1}(P_1^{pf}(X))\subseteq P_1^f(\tilde{X})$. Given $\alpha\in p_*^{-1}(P_1^{pf}(X)) $, we get  $p\alpha\in P_1^f(X)$. Then, for any $\gamma\in \pi_m(A)$ with $m\geq 1$, it holds $p_*[\alpha,f\gamma]=[p\alpha,pf\gamma]=0$. Because $p_*$ is a monomorphism, $[\alpha,f\gamma]=0$ and so $\alpha\in P_1^f(\tilde{X})$ and the proof follows.
\end{proof}

\begin{theorem}\label{PP} If  $p \colon  \tilde{X}\to X$  is a  covering map  and  $f \colon  A\to \tilde{X}$ then the monomorphism $p_\ast \colon  \pi_m(\tilde{X})\to \pi_m(X)$ for $m\geq  1$ yields: 
\begin{enumerate}
\item $G_m^f(\tilde{X})= p_*^{-1}(G_m^{pf}(X))$;
\item $P_m^f(\tilde{X})= p_*^{-1}(P_m^{pf}(X))$.
\end{enumerate}
\end{theorem}

\begin{proof}In view of Lemma~\ref{lemma7}, we may assume  $m>1$.
Because $p_\ast \colon  \pi_m(\tilde{X})\to\pi_m(X)$ is an isomorphism, we derive  that its restrictions \[ p_{\ast|} \colon  G_m^f(\tilde{X})\to  G_m^{pf}(X) \text{ and } p_{\ast|} \colon P_m^f(\tilde{X})\to P_m^{pf}(X)\] are monomorphisms.

First, we prove that $p_{\ast|} \colon  G_m^f(\tilde{X})\to G_m^{pf}(X)$ is surjective for $m>1$.
Given $\alpha\in  G_m^{pf}(X)$, there are $\beta\in\pi_m(\tilde{X})$ such that $p\beta=\alpha$ and $F \colon  A\times \mathbb{S}^m\to X$ extending $pf\vee\alpha \colon  A\vee\mathbb{S}^m\to X$. 
But the $2$-skeleton $(A\times \mathbb{S}^m)^{(2)}=A^{(2)}\vee (\mathbb{S}^m)^{(2)}$, so we get that $\pi_1(A\times \mathbb{S}^m)=\pi_1(A^{(2)}\vee (\mathbb{S}^m)^{(2)})=\pi_1(A^{(2)})=\pi_1(\tilde{A})$. Because $F_{|A\vee \mathbb{S}^m}=pf\vee \alpha$, this implies that $F_*(\pi_1(A\times \mathbb{S}^m))=(pf)_*(\pi_1(A))\subseteq p_*(\pi_1(\tilde{X}))$. Hence, the map $F \colon  A\times \mathbb{S}^m\to X$ lifts to $\tilde F \colon  A\times \mathbb{S}^m\to \tilde{X}$. Because $p\tilde{F}_{|A\vee\mathbb{S}^m}=F_{|A\vee \mathbb{S}^m}=p(f\vee\beta)$, the map $\tilde{F}$ extends $f\vee\beta \colon  A\vee\mathbb{S}^m\to\tilde{X}$ and so $\beta\in G_m^f(\tilde{X})$.

Now, we show that that $ p_{\ast|} \colon  P_m^f(\tilde{X})\to P_m^{pf}(X)$ is surjective for $m>1$.
Given $\alpha\in  P_m^{pf}(X)$, there is $\beta\in\pi_m(\tilde{X})$ such that $p\beta=\alpha$. 
Then, $p_\ast[\beta,f\gamma]=[p\beta,pf\gamma]=[\alpha,pf\gamma]=0$ for any $\gamma\in \pi_m(\tilde{X})$ which implies $[\beta,f\gamma]=0$ and so $\beta\in P_m^f(\tilde{X})$,
and the proof is complete.
\end{proof}

Then, Proposition~\ref{L} and Theorem~\ref{PP} yield:
\begin{corollary}\label{C} If $p \colon  \tilde{X}\to X$ is a covering map then $G_m(\tilde{X})=p_*^{-1}(G_m^p(X))$ and  $P_m(\tilde{X})=p_*^{-1}(P_m^p(X))$ for $m\geq  1$. Further, if $X$ is a simple space then $P_m(\tilde{X})=p_*^{-1}(P_m^p(X))=p_*^{-1}(P_m(X))$ for $m\geq 1$.
\end{corollary}

Given a finite group $H$ acting freely and cellularly on $\mathbb{S}^{2n+1}$  write  $\mathbb{S}^{2n+1}/H$ for the  orbit space. Oprea \cite[\textsc{Theorem~A}]{oprea} has shown that  $G_1(\mathbb{S}^{2n+1}/H) = \mathcal{Z}H$. Then, the relations
 \[\mathcal{Z}H=G_1(\mathbb{S}^{2n+1}/H) \subseteq P_1(\mathbb{S}^{2n+1}/H)\subseteq \mathcal{Z}H\] imply that 
$P_1(\mathbb{S}^{2n+1}/H)= \mathcal{Z}H$, what was already observed by Gottlieb \cite[\textsection3]{gottlieb1} and follows
from Remark~\ref{rem.homoteq}\ref{rem.homoteq.2} as well.

Let $\mathbb{F}$ denote the field of reals $\mathbb{R}$, complex numbers $\mathbb{C}$ or the skew $\mathbb{R}$-algebra of quaternions $\mathbb{H}$ and
$\mathbb{F} P^n$ the appropriate projective $n$-space for $n\ge 1$.
Write $\gamma_n(\mathbb{F})\colon \mathbb{S}^{d(n+1)-1}\to \mathbb{F} P^n$ for the quotient map, where $d=\dim_{\mathbb{R}}\mathbb{F}$ and $i_{\mathbb{F}}\colon \mathbb{S}^d=\mathbb{F}P^1\hookrightarrow \mathbb{F}P^n$ for the canonical inclusion with $n\geq 1$, and recall that \[ \pi_m(\mathbb{F}P^n) = \gamma_n(\mathbb{F})_*\pi_m(\mathbb{S}^{d(n+1)-1})\oplus i_{\mathbb{F}*}E\pi_{m-1}(\mathbb{S}^{d-1}).  \] Given a space $X$ and a prime number $p$, write $\pi_m(X;p)$ for the $p$-primary component of $\pi_m(X)$. Then, the results below are direct consequences of Remark~\ref{rem.homoteq}\ref{rem.homoteq.5} and the result \cite[(4.1-3)]{barratt} (see also \cite[Lemma~2.4]{juno-marek}).

\begin{example} 
\begin{enumerate}\thmenumhspace
\item $G_m^{\gamma_n(\mathbb{F})}(\mathbb{F} P^n)=P_m^{\gamma_n(\mathbb{F})}(\mathbb{F} P^n)=\operatorname{Ker}[\gamma_{n}(\mathbb{F}),-]$ for $m,n\ge 1$.

\item $P_m^{\gamma_n(\mathbb{R})}(\mathbb{R} P^n)=P_m(\mathbb{R} P^n)=\gamma_{n}(\mathbb{R})_*P_m(\mathbb{S}^n)$ for $m> 1$, $n\ge 1$ and
\[P_1^{\gamma_n(\mathbb{R})}(\mathbb{R} P^n)=\begin{cases}\pi_1(\mathbb{R} P^n), &\text{if $n$ is odd;}\\
0, & \text{if $n$ is even.}
\end{cases}\]

\item
\begin{enumerate}
\item  $P_m^{\gamma_n(\mathbb{C})}(\mathbb{C} P^n)=P_m(\mathbb{C} P^n)=\gamma_{n}(\mathbb{C})_*P_m(\mathbb{S}^{2n+1})$ for $m>2$ and $n\geq 1$ odd,  and
\[P_2^{\gamma_n(\mathbb{C})}(\mathbb{C} P^n)=\begin{cases}\pi_2(\mathbb{C} P^n), & \text{if $n$ is odd;}\\
2\pi_2(\mathbb{C} P^n), & \text{if $n$ is even.}
\end{cases}\]

In particular, $P_m^{\eta_2}(\mathbb{S}^2)=\pi_m(\mathbb{S}^2)$ for $m\geq 1$. 

\item  $P_{2n+1}^{\gamma_n(\mathbb{C})}(\mathbb{C} P^n)=P_{2n+1}(\mathbb{C} P^n)=\gamma_{n}(\mathbb{C})_*P_{2n+1}(\mathbb{S}^{2n+1})$ and $2P_m^{\gamma_n(\mathbb{C})}(\mathbb{C} P^n)\subseteq P_m(\mathbb{C} P^n)=\gamma_{n}(\mathbb{C})_*P_m(\mathbb{S}^{2n+1})$ for $m>2n+1$ and $n$ even.

\end{enumerate}
\item
\begin{enumerate}
\item $P_m^{\gamma_n(\mathbb{H})}(\mathbb{H} P^n;p)=\gamma_{n}(\mathbb{H})_*\pi_m(\mathbb{S}^{4n+3};p)\oplus i_{\mathbb{H}\ast}EL''_{m-1}(\mathbb{S}^3;p)$, if $p$ is an odd prime.

\item $P_m^{\gamma_n(\mathbb{H})}(\mathbb{H} P^n;2)=\gamma_{n}(\mathbb{H})_*\pi_m(\mathbb{S}^{4n+3};2)\oplus i_{\mathbb{H}\ast}EL''(\mathbb{S}^3;2)$,  provided  \[[\iota_{4n+3},\pi_m^{4n+3}]\cap (n+1)\nu_{4n+3}E^{4n+3}\pi_{m-1}^3=0,\] where $L''_{m-1}(\mathbb{S}^3)=\{\beta\in\pi_{m-1}(\mathbb{S}^3);\; [i_{\mathbb{H}}E\beta,\gamma_n{(\mathbb{H})}]=0\}$.
\end{enumerate} 
 In particular,  \[ P_m^{\nu_4}(\mathbb{S}^4)=\begin{cases}\nu_{4\ast}\pi_m(\mathbb{S}^{7};p)\oplus EL''_{m-1}(\mathbb{S}^3;p), & \text{if $p$ is an odd prime;}\\
\nu_{4\ast}\pi_m(\mathbb{S}^{7};2)\oplus EL''(\mathbb{S}^3;2).
\end{cases}\]

\item $P_m^{\sigma_8}(\mathbb{S}^8)=\begin{cases}\sigma_{8\ast}\pi_m(\mathbb{S}^{15};p)\oplus EL''_{m-1}(\mathbb{S}^7;p), & \text{if $p$ is an odd prime;}\\
\sigma_{8\ast}\pi_m(\mathbb{S}^{15};2)\oplus EL''(\mathbb{S}^7;2). 
\end{cases}$

\item $G_m^{i_\mathbb{F}}(\mathbb{F}P^n)=P_m^{i_\mathbb{F}}(\mathbb{F}P^n)=\operatorname{Ker}[i_{\mathbb{F}},-]$ for $m,n\geq 1$.

\item $P_m^{i_{\mathbb{R}}}(\mathbb{R} P^n)=\pi_m(\mathbb{R} P^n)$ for $m,n\geq  1$ provided  $n$ is odd, and $P_m^{i_{\mathbb{R}}}(\mathbb{R} P^n)=0$ for $1\leq m\leq n$ provided $n$ is even.

\item 
\begin{enumerate}
\item $P_2^{i_{\mathbb{C}}}(\mathbb{C} P^1)=0$ and  $P_m^{i_{\mathbb{C}}}(\mathbb{C} P^n)=\pi_m(\mathbb{C} P^n)$ for $m>  2$ provided  $n\geq  1$ is odd.
\item  $P_m^{i_{\mathbb{C}}}(\mathbb{C} P^n)=\pi_m(\mathbb{C}P^n)$ for $1\leq m< 2n+1$, $P_{2n+1}^{i_{\mathbb{C}}}(\mathbb{C} P^n)=2\pi_{2n+1}(\mathbb{C}P^n)$ and $P_m^{i_{\mathbb{C}}}(\mathbb{C} P^n)\supseteq 2\pi_m(\mathbb{C}P^n)$ for $m>2n+1$ provided $n$ is even.
\end{enumerate} 

\item $P_4^{i_{\mathbb{H}}}(\mathbb{H} P^1)=0$ and $P_m^{i_{\mathbb{H}}}(\mathbb{H} P^n)=\pi_m(\mathbb{H}P^n)$ for $4< m < 4n+3$, \[P_{4n+3}^{i_{\mathbb{H}}}(\mathbb{H} P^n)=\frac{24}{(n+1,24)}\gamma_n(\mathbb{H})_*\pi_{4n+3}(\mathbb{S}^{4n+3})\oplus i_{\mathbb{H}*}E\pi_{4n+2}(\mathbb{S}^3)\] and 
$P_m^{i_{\mathbb{H}}}(\mathbb{H} P^n)\supseteq \frac{24}{(n+1,24)}\gamma_n(\mathbb{H})_*\pi_m(\mathbb{S}^{4n+3})\oplus i_{\mathbb{H}*}E\pi_{m-1}(\mathbb{S}^3)$ for $m> 4n+3$ and $n\geq 1$.
\end{enumerate}
\end{example}

To extend \cite[\textsc{Theorem~A}]{oprea} and state the main result, we make some prerequisites. Let   $\pi$ be an abelian group. Given $\Phi \colon  A\times \mathbb{S}^1\to X$, write $f=\Phi_{|A}$ for its restriction to the space $A$.
The induced map $\Phi^\ast \colon  H^n(X;\pi)\to H^n( A\times \mathbb{S}^1;\pi)$ on cohomology gives \[\Phi^\ast(x)= f^\ast(x)\otimes 1+ x_\Phi\otimes\lambda,\] where $x\in H^n(X;\pi)$, the element $\lambda$ is a chosen generator of $H^1(\mathbb{S}^1;\mathbb{Z})$ and $x_\Phi\in H^{n-1}(X;\pi)$.

Now, take an integer $m>0$. Recall that a map $p \colon  E\to B$ is called a \textit{principal $K(\pi,m)$-fibration} if it is a pullback of the path fibration $K(\pi,m)\to PK(\pi,m+1)\to K(\pi,m+1)$ via a map $k \colon  B\to K(\pi,m+1)$. If $\iota\in H^{m+1}(K(\pi,m+1);\pi)$ is the characteristic class, let $k^\ast(\iota)=\mu\in H^{m+1}(B;\pi)$ and recall that a map $\varphi \colon  Y\to B$ has a lifting $\bar{\varphi} \colon  Y\to E$ if and only if $\varphi^\ast(\mu)=0$.

Then, following \textit{mutatis mutandis}  the proof of \cite[\textsc{Theorem~1}]{oprea}, we can show a fundamental lifting result due to Gottlieb \cite{gottlieb}:

\begin{lemma}\label{Ll} Let $p \colon  E\to B$ and  $p' \colon E' \to B' $ be  principal $K(\pi,m)$- and  $K(\pi',m)$-fibrations, respectively with a commutative diagram 
\[\xymatrix@C=1cm{E\ar[d]_{p}\ar[r]^{\bar{f}}&E'\ar[d]^{p'}\\
B \ar[r]_{f}&B'}\]
and  $\Phi \colon B\times  \mathbb{S}^1 \to B'$ such that $\Phi_{|B}=f$. Then, there exists a map $\bar{\Phi} \colon  E\times \mathbb{S}^1\to E'$ such that $\bar \Phi_{|E}=\bar f$ and the diagram
\[\xymatrix@C=1.5cm{E\times \mathbb{S}^1\ar[d]_{ p\times \mathrm{id}_{\mathbb{S}^1} }\ar@{-->}[r]^-{\bar{\Phi}}&E'\ar[d]^{p'}\\
B\times \mathbb{S}^1 \ar[r]_-{\Phi}&B'}\] commutes
if and only if $\mu'_\Phi=0\in H^m(B;\pi')$, where $\Phi^\ast(\mu')= \phi^\ast(\mu')\otimes 1+ \mu'_\Phi\otimes\lambda$ and $\lambda$ is a chosen generator of $H^1(\mathbb{S}^1;\mathbb{Z})$.
\end{lemma}

\begin{remark} As in \cite[\textsc{Remark}, p.\ 68]{oprea}, we notice that without loss of generality we can take the diagrams above to be homotopy commutative.
\end{remark}

Next, given a space $X$, consider its universal covering $p \colon  \tilde{X}\to X$. As usual, we can take its classifying map $X\to K(\pi_1(X),1)$ to be an inclusion.
If $\pi_1(X)$ acts trivially on $\pi_m(X)$ for $m>1$ then  the pair $(K(\pi_1(X),1),X)$ is simple. Hence, according to \cite[Chapter 8, \textsection 3]{spanier}, the  Moore--Postnikov tower

\begin{equation*}
\xymatrix@C=3.5cm@R=.5cm{
 & \vdots \ar[d] \\ &  X(m+1) \ar[d]^{p_{X(m+1)}} \\ & X(m) \ar[d] \\ & \vdots \ar[d] \\ &  X(1) \ar[d] \\ X\ar[r] \ar[ru] \ar[ruuu] \ar[ruuuu]^{q_{X(m+1)}} & X(0) \rlap {${}= K(\pi_1(X),1)$} 
}
\end{equation*}
for the classifying map  $X\to K(\pi_1(X),1)$ of the covering $p$ exists, where  $X(m)$  is called the $m$\textsuperscript{th} stage of this tower for $m\geq 0$.

\medskip
From now on, we assume that $A$ and $X$ are spaces  such that $\pi_1(A)$ and $\pi_1(X)$ act trivially on $\pi_m(A)$ and $\pi_m(X)$ for $m>1$, respectively. Given a map $f\colon A\to X$, write $f(m)\colon A(m)\to X(m)$ for the induced map of the $m$\textsuperscript{th} stages for $m\ge 0$.

If $\alpha\in G_1^{f(m+1)}(X(m+1))$ and $\bar \Phi\colon A(m+1)\times \mathbb{S}^1\to X(m+1)$ is the associated map then naturality of Postnikov system provides a homotopy commutative diagram \[\xymatrix@C=1.5cm{ A(m+1)\times \mathbb{S}^1 \ar[d]_{p_{A(m+1)} \times \mathrm{id}_{\mathbb{S}^1}}\ar[r]^-{\bar \Phi}
& X(m+1)\ar[d]^{p_{X(m+1)}}\\
 A(m)\times \mathbb{S}^1 \ar@{-->}[r]_-{\Phi}&X(m) \rlap{.}}\]
Taking $p_{A(m+1)} \colon  A(m+1)\to A(m) $ to be an inclusion, the obstructions to the existence of a relative homotopy from $\Phi_{|A(m)}$ to $f(m)$ lie in $H^i(A(m),A(m+1);\pi_i(X(m)))=0$. Hence, $\Phi_{|A(m)}$ is homotopic to $f(m)$ and $\Phi$ is the associated map to $p_{X(m+1)*}(\alpha)\in G_1^{f(m)}(X(m))$.

Then, Lemma~\ref{Ll} leads to the following generalization of \cite[Lemma~4]{haslam} and \cite[\textsc{Theorem~2}]{oprea}:
\begin{proposition}The map $p_{X(m+1)}\colon X(m+1)\to X(m)$ implies a homomorphism $p_{X(m+1)*}\colon G_1^{f(m+1)}(X(m+1))\to G_1^{f(m)}(X(m))$.
\end{proposition}

Now, for  $\alpha\in \mathcal{Z}_{\pi_1(X)}f_\ast(\pi_1(A))$, consider the map $\Phi_\alpha \colon  K(\pi_1(A),1)\times \mathbb{S}^1 \to K(\pi_1(X),1)$ corresponding to the homomorphism \[\pi_1(A)\times \mathbb{Z}\to \pi_1(X)\]
given by $(g,n)\mapsto \alpha^nf_\ast(g)$ for $(g,n)\in \pi_1(A)\times \mathbb{Z}$.

If   $H^m(A(m-1);\pi_m(X))=0$ for $m>1$ then Lemma~\ref{Ll} leads to commutative diagrams \[\xymatrix@C=2cm{ A(m+1)\times \mathbb{S}^1 \ar[d]_{p_{A(m+1)} \times \mathrm{id}_{\mathbb{S}^1}}\ar[r]^-{\Phi_\alpha{(m+1)}}
& X(m+1)\ar[d]^{p_{X(m+1)}}\\
 A(m)\times \mathbb{S}^1 \ar[r]_-{\Phi_\alpha(m)}&X(m)}\]
with $\Phi_\alpha(m)_{|A(m)}= f(m)$. Hence, we obtain a map \[   \lim\limits_\leftarrow \Phi_\alpha(m)\colon \lim\limits_\leftarrow A(m)\times \mathbb{S}^1\to \lim\limits_\leftarrow X(m).\] Let $\phi(A)\colon A\to  \lim\limits_\leftarrow A(m)$
and $\phi(X)\colon X\to  \lim\limits_\leftarrow X(m)$ denote the standard weak homotopy equivalences.
Then,  there is a unique (up to homotopy) map $\hat{\Phi}_\alpha\colon  A\times \mathbb{S}^1 \to  X$ which makes the diagram
\[\xymatrix@C=2.5cm{A\times \mathbb{S}^1\ar[d]_{ \phi(A)\times \mathrm{id}_{\mathbb{S}^1} }\ar@{-->}[r]^-{\hat{\Phi}_\alpha}
& X \ar[d]^{\phi(X)}\\
\lim\limits_\leftarrow A(m)\times \mathbb{S}^1\ar[r]_-{\lim\limits_\leftarrow\Phi_\alpha(m)}&\lim\limits_\leftarrow X(m)}\]
commutative. Certainly, $\hat\Phi_{\alpha|\mathbb{S}^1}=\alpha$. To see that  $\hat\Phi_{\alpha| A}=f$, observe that $\dim A=d<\infty$ provides a bijection of homotopy classes $[A,X]\cong[A,X(d)]\cong[A(d),X(d)]$ and $\hat\Phi_{\alpha| A}$ corresponds to $\Phi_\alpha(d)_{|A(d)}= f(d)$. Then, this bijection implies $\hat\Phi_{\alpha| A}= f$ and  we may state:
\begin{proposition}\label{ll} If $f \colon A\to X$ with $\dim A=d<\infty$,
$\alpha \in \mathcal{Z}_{\pi_1(X)}f_\ast(\pi_1(A))$ and $H^m(A(m-1);\pi_m(X))=0$ for $m>1$ then 
 $\alpha\in G_1^f(X)$.
\end{proposition}

\begin{example}
\begin{enumerate}\thmenumhspace
\item If $X=K(\pi,1)$ for some group $\pi$ then $H^m(A(m-1);\pi_m(X))=0$ for $m>1$. 
\item If $\dim A=1$ then there exists a homotopy equivalence $A\simeq K(\pi,1)$, where $\pi$ is a free group. Hence $A(m-1)=K(\pi,1)$  and consequently $H^m(A(m-1);\pi_m(X))=H^m(\pi;\pi_m(X))=0$ for $m>1$. 
\end{enumerate}

\end{example}

Moreover,  under the conditions on spaces $A$ and $X$ considered in Proposition~\ref{ll}, we are in a position to state the main result of this section:
\begin{theorem}\label{main} If $f \colon  A\to X$ with $1<\dim A=d<\infty$ 
then $G_1^f(X)=\mathcal{Z}_{\pi_1(X)}f_\ast(\pi_1(A))$.
\end{theorem}

\begin{proof}
We know that $G_1^f(X)\subseteq\mathcal{Z}_{\pi_1(X)}f_\ast(\pi_1(A))$, so we have to reverse the inclusion. Let $\alpha\in \mathcal{Z}_{\pi_1(X)}f_\ast(\pi_1(A))$ and consider the corresponding map $\Phi_\alpha \colon   K(\pi_1(A),1)\times \mathbb{S}^1\to K(\pi_1(X),1)$.
Lemma~\ref{Ll} assures that an obstruction to a lifting of $\Phi_\alpha(m-1) \colon   A(m-1)\times \mathbb{S}^1\to X(m-1)$  to
${\Phi}_\alpha(m) \colon  A(m)\times \mathbb{S}^1\to X(m)$ as in the diagram
\[\xymatrix@C=2cm{A(m)\times \mathbb{S}^1 \ar[d]_{p_{A(m)}\times \mathrm{id}_{\mathbb{S}^1}}\ar@{-->}[r]^-{{\Phi}_\alpha(m)}&X(m)\ar[d]^{p_{X(m)}}\\
 A(m-1)\times \mathbb{S}^1 \ar[r]_-{\Phi_{\alpha}(m-1)}&X(m-1)}\]
lies in $H^{m}(A(m-1);\pi_m(X))$ for $m>1$. But,  $H^m(A(m-1);\pi_m(X))=0$ for $1<m\leq n$, and so, in view of  Proposition~\ref{ll}, we have to show that 
$H^{m}(A(m-1);\pi_m(X))=0$ for $m>n$. 

Let $m>n$. Then, the Whitehead theorem shows that the map $q_{A(m-1)}\colon  A\to A(m-1)$ induces  isomorphisms $H_i(A;\mathbb{Z})\xrightarrow{\approx} H_i(A(m-1);\mathbb{Z})$ for $i\leq m-1$ and a surjection $H_{m}(A;\mathbb{Z})\to H_{m}(A(m-1);\mathbb{Z})$. Hence $H_n(A(n);\mathbb{Z})$ is free, $H_{m-1}(A(m-1);\mathbb{Z})=0$ for $m>n+1$ and $H_m(A(m-1);\mathbb{Z})=0$ for $m>n$. The universal coefficient theorem gives
\begin{align*}
H^{m}(A(m-1);\pi_m(X) ) &\approx \operatorname{Hom}(H_{m}(A(m-1));\pi_m(X))\\ &\oplus \operatorname{Ext}(H_{m-1}(A(m-1));\pi_{m}(X))=0 
\end{align*} for $m>n$.
Hence, Proposition~\ref{ll} implies $\alpha\in G_1^f(X)$ and the proof follows.
\end{proof}

Since $G_1^f(X)\subseteq P_1^f(X) \subseteq \mathcal{Z}_{\pi_1(X)}f_*(\pi_1(A))$, we derive that $P_1^f(X) = \mathcal{Z}_{\pi_1(X)}f_*(\pi_1(A))$ under the hypothesis of Proposition~\ref{ll}.

Recall that a finite dimensional $CW$-complex $\Sigma(n)$ with the homotopy type of the $n$-sphere $\mathbb{S}^n$ is called an
$n$-{\em homotopy sphere} for $n \geq 1$. The finite periodic groups, being the only finite groups acting freely on some homotopy sphere, have been fully classified by Suzuki--Zassenhaus, see e.g., \cite[Chapter~IV, Theorem~6.15]{adem}.  It is well-known that the only finite groups acting freely on $\Sigma(2n)$ are $\mathbb{Z}_2$ and the trivial group.

Given a free and cellular action of a finite group $H$ on a homotopy sphere $\Sigma(2m+1)$, write $\gamma_{2n+1}\colon \Sigma(2n+1)\to \Sigma(2n+1)/H$ for the quotient map.
Following \cite[Chapter~VII, Proposition~10.2]{brown}, the action of $H$ on  $\pi_i(\Sigma(2m+1)/H)$ is trivial for $i>1$. Because $H^{2i+1}(H;\mathbb{Z})=0$ for $i\geq 1$ (see e.g., \cite[\textsc{Lemma~6}]{oprea}), Theorem~\ref{main} yields the following generalization of \cite[\textsc{Theorem~A}]{oprea}:
\begin{corollary}  If  $f \colon  A\to \Sigma(2n+1)/H$ with $1<\dim A=d\le 2n+1$, $\pi_i(A)=0$ for $1<i<2n+1$
and $H^{2n+1}(\pi_1(A);\mathbb{Z})=0$ then $G_1^f(\Sigma(2n+1)/H)=\mathcal{Z}_Hf_\ast(\pi_1(A))$.

In particular, if $f \colon  \Sigma(2d+1)/K\to \Sigma(2n+1)/H$ and $d\leq n$ then $G_1^f(\Sigma(2n+1)/H)=\mathcal{Z}_Hf_\ast(K)$.
\end{corollary}

\begin{proof}
The first part is a direct conclusion from Theorem~\ref{main}.

Let now $f\colon \Sigma(2d+1)/K\to \Sigma(2n+1)/H$ with $d\leq n$. Notice that by \cite{adem-davis} or \cite{wall} there is a homotopy sphere $\Sigma'(2d+1)$ admitting a free action of the group $K$ such that $\dim \Sigma'(2d+1)=2d+1$ and the space forms $\Sigma(2d+1)/K$ and $\Sigma'(2d+1)/K$ are homotopy equivalent. 

Let $d=n$. Because $H^{2n+1}(K;\mathbb{Z})=0$ (see e.g., \cite[\textsc{Lemma~6}]{oprea}), the space $\Sigma'(2d+1)/K$ satisfies all required hypotheses and, in view of Remark~\ref{rem.homoteq}, the proof follows.

Let now $d<n$. Then, $\Sigma'(2d+1)/K$ satisfies all hypotheses  of Theorem~\ref{main} and, in view of Remark~\ref{rem.homoteq}, the proof is complete.
\end{proof}

Let $f \colon  A\to \Sigma(2n+1)/H$ satisfies the conditions above. First, notice that \[P_1^f(\Sigma(2n+1)/H)=G_1^f(\Sigma(2n+1)/H)=\mathcal{Z}_Hf_\ast(\pi_1(A)),\] 
\[P_1^f(\Sigma(2n+1)/H)=G_1^f(\Sigma(2n+1)/H)=H\] provided $A$ is $1$-connected.
This implies \[P_1^f(\Sigma(2n+1)/H)=G_1^f(\Sigma(2n+1)/H)=H\] for any map $f \colon \Sigma(2d+1)\to \Sigma(2n+1)/H$ and $d,n\geq 1$.
Further, \[P_1^f(\Sigma(2n+1)/H)=G_1^f(\Sigma(2n+1)/H)=P_1(\Sigma(2n+1)/H)=G_1(\Sigma(2n+1)/H)=H\] for $f\colon A \to  \Sigma(2n+1)/H$ satisfying the condition above and $H$ is abelian.
In particular, in view of Theorem~\ref{PP}, we derive:

\begin{corollary}\label{CCC}
\begin{enumerate}\thmenumhspace
\item $P_1^{\gamma_{2n+1}}(\Sigma(2n+1)/H)=G_1^{\gamma_{2n+1}}(\Sigma(2n+1)/H)=H$;

\item $G_m^{\gamma_{2n+1}}(\Sigma(2n+1)/H)=\gamma_{{2n+1}\ast}G_m(\Sigma(2n+1))$ for $m>1$;
\item  $P_m^{\gamma_{2n+1}}(\Sigma(2n+1)/H)=\gamma_{{2n+1}\ast}P_m(\Sigma(2n+1))=\gamma_{{2n+1}\ast}G_m(\Sigma(2n+1))$ for $m>1$.
\end{enumerate}
\end{corollary}

Further, the result \cite[Theorem~2.3]{dac-marek} yields:

\begin{corollary} If $p$ is a prime not dividing the order  of the finite group $H$ then \[\gamma_{2n+1*}(G_m(\Sigma({2n+1});p))
= G_m(\Sigma({2n+1})/H;p)=\pi_m(\Sigma({2n+1})/H;p)\] for $m>1$.
\end{corollary}

\section{Gottlieb groups of spaces $\Sigma(2n+1)/\mathbb{Z}_l$}\label{sec.2}

First, recall that given a  free action $\mathbb{Z}_2\times \Sigma(n)\to \Sigma(n)$, by \cite[Lemma~2.5]{kwasik}, there is a homotopy equivalence $\Sigma(n)/\mathbb{Z}_2\simeq \mathbb{R}P^n$.
 Then, $G_m(\Sigma(n)/\mathbb{Z}_2)\approx G_m(\mathbb{R}P^n)$  and $G_m(\Sigma(n))\approx G_m(\mathbb{S}^n)$ for $m\geq 1$. Hence,  {\cite{pak}} implies:

\begin{theorem}\label{PW}$G_{2n+1}(\Sigma(2n+1)/\mathbb{Z}_2) ={\gamma_{2n+1}}_\ast G_{2n+1}(\Sigma(2n+1))$
\begin{align*}
&= \begin{cases}
\pi_{2n+1}(\Sigma(2n+1)/\mathbb{Z}_2), & \text{for $n=0,1,3$;}\\
2\pi_{2n+1}(\Sigma(2n+1)/\mathbb{Z}_2), & \text{for odd $n$ and $n\neq 0,1,3$.}
\end{cases}
\end{align*}
\end{theorem}

In virtue of the inclusion $G_m(\Sigma(2n+1)/\mathbb{Z}_2)\subseteq {\gamma_{2n+1}}_*G_m(\Sigma(2n+1))$ for $m>1$
and the group ${\gamma_{2n+1}}_*G_m(\Sigma(2n+1))$ is an upper bound of  $G_m(\Sigma(2n+1)/\mathbb{Z}_2)$.
Now, \cite[Theorem~2.41]{juno-marek} implies:

\begin{theorem}
If $m\le 7$ then the equality \[G_{m+2n+1}(\Sigma(2n+1)/\mathbb{Z}_2)={\gamma_{2n+1}}_\ast G_{m+2n+1}(\Sigma(2n+1))\] holds
except for the following pairs: $(m,2n+1)=(3,2^l-3)$ with $l\ge 4$, $(6,2^l-5)$ with $l\geq 5$ and $(7,11)$. Furthermore:
\begin{enumerate}
\item $G_{2^l}(\Sigma(2^l-3)/\mathbb{Z}_2)\supseteq 2\pi_{2^l}(\Sigma(2^l-3)/\mathbb{Z}_2)$ for $l\ge 4$;

\item $G_{18}(\Sigma(11)/\mathbb{Z}_2)\supseteq 2\pi_{18}(\Sigma(11)/\mathbb{Z}_2)$.
\end{enumerate}
\end{theorem}

\begin{remark}\label{mukai}
According to the J.\ Mukai's conjecture, \[G_{18}(\Sigma(11)/\mathbb{Z}_2)= 2\pi_{18}(\Sigma(11)/\mathbb{Z}_2)=2\gamma_{18*}\pi_{18}(\Sigma(11))=2\gamma_{18*}G_{18}(\Sigma(11)).\] This implies that there is the proper inclusion $G_{18}(\Sigma(11)/\mathbb{Z}_2)\subsetneqq \gamma_{18*}G_{18}(\Sigma(11))$.
\end{remark}

Next, \cite[Proposition~2.42]{juno-marek} leads to:
\begin{proposition}\label{8910}
$G_{m+2n+1}(\Sigma(2n+1)/\mathbb{Z}_2)=\pi_{m+2n+1}(\Sigma(2n+1)/\mathbb{Z}_2)$ if $n\equiv 1 \pmod 2$ and $m=8,9$,
 $n\equiv 1 \pmod 2$ with $n\ge 3$ and $m=10$,  $n\equiv 0 \pmod 2$, $n \equiv 3 \pmod 4$ with $n\ge 6$ and $m=11$, or
$n\ge 7$ and $m=13$.
\end{proposition}

To study $G_m(\mathbb{S}^{2n+1}/\mathbb{Z}_l)$ for $l>2$ we need the following. Let $H$ be a closed subgroup of a Lie group $K$ and  write $K/H$ for the associated orbit space.
Then, \cite[Theorem~II.5]{lang} yields:
\begin{lemma}\label{SL}
The projection map $p\colon  K\to K/H$  leads to $p_*(\pi_m(K))\subseteq G_m(K/H)$ for $m\geq 1$.
\end{lemma}

If $H$ is a finite subgroup of $K$ then Proposition~\ref{l1} and Lemma~\ref{SL} imply $p_*(\pi_m(K))= G_m(K/H)$ for $m> 1$. In particular, if $H$ is a finite subgroup of $\mathbb{S}^3$ then $p_*(\pi_m(\mathbb{S}^3))= G_m(\mathbb{S}^3/H)$ for $m> 1$.

Let now $U(n)$ be the $n$\textsuperscript{th} unitary group and write $i'_n\colon  U(n)\times \mathbb{Z}_l\hookrightarrow U(n)\times U(1)\hookrightarrow U(n+1)$ for the canonical inclusion map.
Then, we may identify $U(n+1)/U(n)=\mathbb{S}^{2n+1}$ and $U(n+1)/U(n)\times \mathbb{Z}_l=\mathbb{S}^{2n+1}/\mathbb{Z}_l$.
Set $p_n \colon  U(n+1)\to \mathbb{S}^{2n+1}$ and $p'_n\colon  U(n+1)\to \mathbb{S}^{2n+1}/\mathbb{Z}_l$ for the appropriate  quotient maps.
Next,  consider the map of exact sequences induced by the fibrations $U(n+1)\xrightarrow{U(n)} \mathbb{S}^{2n+1}$ and
$U(n+1)\xrightarrow{U(n)\times \mathbb{Z}_l} \mathbb{S}^{2n+1}/\mathbb{Z}_l$:
\begin{equation*}\label{comg1}
\begin{gathered}
\xymatrix@C=1cm{
\pi_m(U(n+1))\ar@{=}[d]\ar[r]^-{p_\ast}&
\pi_m(\mathbb{S}^{2n+1})\ar[d]^{{\gamma_{2n+1}}_\ast}\ar[r]^-{\Delta_{\mathbb{C}}}&\pi_{m-1}(U(n))\ar@{^{(}->}[d]\\
\pi_m(U(n+1))\ar[r]_-{p'_\ast}&\pi_m(\mathbb{S}^{2n+1}/\mathbb{Z}_l)\ar[r]_-{\Delta'_{\mathbb{C}}}&\pi_{m-1}(U(n)\times \mathbb{Z}_l) \rlap{.} }
\end{gathered}\end{equation*}
Then, by Lemma~\ref{SL}, we have:
\begin{equation*}\label{sl1}
\operatorname{Ker} \Delta'_\mathbb{C}=p'_*\pi_m(U(n+1))\subseteq G_m(\mathbb{S}^{2n+1}/\mathbb{Z}_l).
\end{equation*}

Let $J\colon \pi_m(SO(n))\to \pi_{m+n}(\mathbb{S}^n)$ be the $J$-homomorphism and  $r\colon U(n) \hookrightarrow SO(2n)$ be the canonical inclusion. Taking $J_{\mathbb{C}}=J\circ r_*$, the result \cite[Lemma~2.33]{juno-marek} implies:

\begin{proposition}\label{fund1}
\begin{enumerate}\thmenumhspace
\item  $\operatorname{Ker} \bigl( \Delta_\mathbb{C}\colon  \pi_m(\mathbb{S}^{2n+1})\to\pi_{m-1}(U(n))\bigr)
\subseteq \gamma_{2n+1\ast}^{-1} G_m(\mathbb{S}^{2n+1}/\mathbb{Z}_l)$.
In particular, it holds $\gamma_{2n+1\ast}\pi_m(\mathbb{S}^{2n+1};p)\subseteq G_m(\mathbb{S}^{2n+1}/\mathbb{Z}_l)$
provided $\pi_{m-1}(U(n);p)=0$ for a prime $p$.

\item  Let $m\ge 3$. If $E\circ {J_\mathbb{C}}_{|\Delta_{\mathbb{C}}(\pi_m(\mathbb{S}^{2n+1}))} \colon  \Delta_{\mathbb{C}}(\pi_m(\mathbb{S}^{2n+1}))\to \pi_{m+2n}(\mathbb{S}^{2n+1})$
is a monomorphism then $\gamma_{2n+1\ast}G_m(\mathbb{S}^{2n+1})\subseteq G_m(\mathbb{S}^{2n+1}/\mathbb{Z}_l)$.
In particular, under the assumption before, $G_m(\mathbb{S}^{2n+1}/\mathbb{Z}_l)=\gamma_{2n+1\ast} G_m(\mathbb{S}^{2n+1})$.
\end{enumerate}
\end{proposition}

Next, we make use of
\cite[Theorem~2.45]{juno-marek} to get:
\begin{proposition}\label{GCP}
\begin{enumerate}\thmenumhspace
\item Let $m=1,2$. Then,
\begin{align*}
G_{m+2n+1}(\mathbb{S}^{2n+1}/\mathbb{Z}_l)
&=
\begin{cases}
0, & \text{if $n$ is even,}\\
\pi_{m+2n+1}(\mathbb{S}^{2n+1}/\mathbb{Z}_l)\approx\mathbb{Z}_2, & \text{if $n$ is odd.}
\end{cases}
\end{align*}

\item $G_{2n+4}(\mathbb{S}^{2n+1}/\mathbb{Z}_l)\supseteq$
\begin{align*}
 \begin{cases}
(24,n)\pi_{2n+4}(\mathbb{S}^{2n+1}/\mathbb{Z}_l)\approx\mathbb{Z}_{\frac{24}{(24,n)}}, & \text{if $n$ is even,}\\
\frac{(24,n+3)}{2}\pi_{2n+4}(\mathbb{S}^{2n+1}/\mathbb{Z}_l)\approx\mathbb{Z}_{\frac{48}{(24,n+3)}},
& \text{if $n$ is odd.}
\end{cases}
\end{align*}
In particular,  $G_{2n+4}(\mathbb{S}^{2n+1}/\mathbb{Z}_l)=2\pi_{2n+4}(\mathbb{S}^{2n+1}/\mathbb{Z}_l)$ for 
$n\equiv 2,10 \pmod{12}$ with $n\geq 10$
except $n=2^{i-1}-2$ or $n\equiv 1,17\pmod{24}$ with $n\ge 17$, and $G_{2n+4}(\mathbb{S}^{2n+1}/\mathbb{Z}_l)=\pi_{2n+4}(\mathbb{S}^{2n+1}/\mathbb{Z}_l)$ for  $n\equiv 7,11\pmod{12}$.

\item $G_{2n+6}(\mathbb{S}^{2n+1}/\mathbb{Z}_l)
=\pi_{2n+6}(\mathbb{S}^{2n+1}/\mathbb{Z}_l)
\approx
\begin{cases}
0, & \text{for $n\geq 3$,}\\
\mathbb{Z}_2, & \text{for $n=2$.}
\end{cases} $

\item $G_{2n+7}(\mathbb{S}^{2n+1}/\mathbb{Z}_l)=\pi_{2n+7}
(\mathbb{S}^{2n+1}/\mathbb{Z}_l)$ for $n\equiv 2,3\pmod 4$.
\end{enumerate}
\end{proposition}

Further, the result \cite[Proposition~2.46]{juno-marek} yields:
\begin{proposition}
$G_m(\mathbb{S}^{7}/\mathbb{Z}_l;2)=\pi_m(\mathbb{S}^{7}/\mathbb{Z}_l;2)$ for $8\leq m\leq 24$ unless $m=15,21$.
Further, $\gamma_{7\ast}\{\sigma'\eta_{14},\bar{\nu}_7+\varepsilon_7\}\subseteq G_{15}(\mathbb{S}^{7}/\mathbb{Z}_l;2)$ and
$\gamma_{7\ast}\{\sigma'\sigma_{14}\}\subseteq G_{21}(\mathbb{S}^{7}/\mathbb{Z}_l;2)$.
\end{proposition}

Then, as in \cite[Corollary~2.47]{juno-marek}, we deduce:
\begin{corollary}\label{GCP1}
\begin{enumerate}\thmenumhspace
\item 
$G_m(\mathbb{S}^{5}/\mathbb{Z}_l)=\pi_m(\mathbb{S}^{5}/\mathbb{Z}_l)$ for $10\leq m\leq 12$;

\item  $G_m(\mathbb{S}^{5}/\mathbb{Z}_l;p)=\pi_m(\mathbb{S}^{5}/\mathbb{Z}_l;p)$ for an odd prime $p$;

\item $G_m(\mathbb{S}^{4n+3}/\mathbb{Z}_l)\supseteq(\gamma_{4n+3}\eta^m_{4n+3})_*\pi_m^{4n+m+3}$ for $m=1,2$;

\item
\begin{enumerate}[label={\textup{(\roman{*})}}]
\item 
$G_m(\mathbb{S}^{4n+1}/\mathbb{Z}_l)\supseteq 2(12,n)(\gamma_{2n}\nu^+_{4n+1})_*\pi_m(\mathbb{S}^{4n+4})$;

\item $G_m(\mathbb{S}^{4n+3}/\mathbb{Z}_l)\supseteq (12,n+2)(\gamma_{2n+1}\nu^+_{4n+3})_*\pi_m(\mathbb{S}^{4n+6})$;
\end{enumerate}

\item $G_m(\mathbb{S}^{8n+5}/\mathbb{Z}_l)\supseteq(\gamma_{4n+2}\nu^2_{8n+5})_*\pi_m^{8n+11}$;

\item $G_{8n+11}(\mathbb{S}^{8n+3}/\mathbb{Z}_l)=\pi_{8n+11}(\mathbb{S}^{8n+3}/\mathbb{Z}_l)$ for  $n\geq 2$;

\item $G_{8n+m}(\mathbb{S}^{8n+7}/\mathbb{Z}_l)=\pi_{8n+m}(\mathbb{S}^{8n+7}/\mathbb{Z}_l)$ for  $m=28,29$ with $n\geq 2$.
\end{enumerate}
\end{corollary}

\section{Gottlieb groups of space forms $\mathbb{S}^{2n+1}/H$}\label{sec.3}

Let $H\subseteq SO(2n + 2)$ be a finite subgroup of the special
orthogonal group $SO(2n+2)$ acting freely on the $(2n+1)$-sphere
$\mathbb{S}^{2n+1}$. We point out that those groups have been fully classified by Wolf in \cite{wolf}.
The orbit spaces $\mathbb{S}^{2n+1}/H$ are called \textit{spherical $(2n+1)$-manifolds} and sometimes also \textit{elliptic 
$(2n+1)$-manifolds} or \textit{Clifford--Klein manifolds}.

Certainly, a spherical $(2n+1)$-manifold $\mathbb{S}^{2n+1}/H$ has a finite fundamental group isomorphic to $H$ itself. The \textit{William Thurston's elliptization conjecture}, proved by Grigori Perelman, states that conversely for $n=1$ all $3$-manifolds with finite fundamental group are spherical manifolds.

Now, write $\gamma_{2n+1} \colon \mathbb{S}^{2n+1}\to \mathbb{S}^{2n+1}/H$ for the quotient map.  Then, in view of Proposition~\ref{l1}, Corollary~\ref{CCC} and Lemma~\ref{SL}, we may state:
\begin{remark}Let $m>1$. Then: 
\begin{enumerate}
\item $G_m^{\gamma_{2n+1}}(\mathbb{S}^{2n+1}/H)=P_m^{\gamma_{2n+1}}(\mathbb{S}^{2n+1}/H)=\gamma_{2n+1*}G_m(\mathbb{S}^{2n+1})$.

\item  $G_m^{\gamma_3}(\mathbb{S}^3/H)=P_m^{\gamma_3}(\mathbb{S}^3/H)=G_m(\mathbb{S}^3/H)=P_m(\mathbb{S}^3/H)=\gamma_{3*}\pi_m(\mathbb{S}^3)$.
\end{enumerate}
\label{lenss}\end{remark}
 
Next, considering the fibration
\[\mathbb{S}^{2n+1}/H \to V_{2n+3,2}/H \to \mathbb{S}^{2n+2}\] for $n\geq 1$,
where $V_{2n+3,2}$ is the Stiefel manifold, it was deduced in \cite{dac-marek} that
\[G_{2n+1}(\mathbb{S}^{2n+1}/H) = \begin{cases}
\mathbb{Z}, & \text{for $n=0,1,3$};\\
2\mathbb{Z}, & \text{for any other $n$.}
\end{cases}
\]

Now, let $Sp(n)$ be the $n$\textsuperscript{th} symplectic group and $H<\mathbb{S}^3=Sp(1)$ be a finite subgroup. Write $i'_n\colon  Sp(n)\times H\hookrightarrow Sp(n)\times Sp(1)
\hookrightarrow Sp(n+1)$ for the canonical inclusion map. Then, we may identify $Sp(n+1)/Sp(n)=\mathbb{S}^{4n+3}$ and $Sp(n+1)/Sp(n)\times H=\mathbb{S}^{4n+3}/H$.
Set $p_n \colon  Sp(n+1)\to \mathbb{S}^{4n+3}$ and $p'_n\colon  Sp(n+1)\to \mathbb{S}^{4n+3}/H$ for the quotient maps.
Now, we consider the map of exact sequences induced by the fibrations $Sp(n+1)\xrightarrow{Sp(n)} \mathbb{S}^{4n+3}$ and
$Sp(n+1)\xrightarrow{Sp(n)\times H} \mathbb{S}^{4n+3}/H$:
\begin{equation*}\label{comg3}
\begin{gathered}
\xymatrix@C=1cm{
\pi_m(Sp(n+1))\ar@{=}[d]\ar[r]^-{p_\ast}&
\pi_m(\mathbb{S}^{4n+3})\ar[d]^{{\gamma_{4n+3}}_\ast}\ar[r]^-{\Delta_{\mathbb{H}}}&\pi_{m-1}(Sp(n))\ar@{^{(}->}[d]\\
\pi_m(Sp(n+1))\ar[r]_-{p'_\ast}&\pi_m(\mathbb{S}^{4n+3}/H)\ar[r]_-{\Delta'_{\mathbb{H}}}&\pi_{m-1}(Sp(n)\times H)\rlap{.}}
\end{gathered}
\end{equation*}
Then, by Lemma~\ref{SL}, we have:
\begin{equation*}\label{sl3}
\operatorname{Ker} \Delta'_\mathbb{H}=p'_*\pi_m(Sp(n+1))\subseteq G_m(\mathbb{S}^{4n+3}/H).
\end{equation*}

Now, let $c\colon Sp(n) \hookrightarrow SU(2n)$ be the canonical inclusion. Taking $J_{\mathbb{H}}=J\circ r_*\circ c_*$, by the result \cite[Lemma~2.33]{juno-marek}, we obtain a key fact determining $G_m(\mathbb{S}^{4n+3}/H)$:
\begin{proposition}\label{fund2}
\begin{enumerate}\thmenumhspace
	\item  $\operatorname{Ker} \bigl(\Delta_\mathbb{H}\colon  \pi_m(\mathbb{S}^{4n+3})\to\pi_{m-1}(Sp(n))\bigr)
\subseteq \gamma_{4n+3\ast}^{-1} G_m(\mathbb{S}^{4n+3}/H)$.
In particular, it holds $\gamma_{4n+3\ast}\pi_m(\mathbb{S}^{4n+3};p)\subseteq G_m(\mathbb{S}^{4n+3}/H)$
provided $\pi_{m-1}(Sp(n);p)=0$ for a prime $p$.

\item  Let $m\ge 5$. If $E^3\circ {J_\mathbb{H}}_{|\Delta_{\mathbb{H}}(\pi_m(\mathbb{S}^{4n+3}))} \colon  \Delta_{\mathbb{H}}(\pi_m(\mathbb{S}^{4n+3}))\to \pi_{m+4n+2}(\mathbb{S}^{4n+3})$
is a monomorphism then $\gamma_{4n+3\ast}G_m(\mathbb{S}^{4n+3})\subseteq G_m(\mathbb{S}^{4n+3}/H)$. In particular, under the assumption before, $G_m(\mathbb{S}^{4n+3}/H)={\gamma_{4n+3}}_\ast G_m(\mathbb{S}^{4n+3})$.
\end{enumerate}
\end{proposition}

By the result \cite[Theorem~2.49]{juno-marek}, we obtain:
\begin{theorem}\label{hthm}
$G_m(\mathbb{S}^{4n+3}/H)\supseteq(24,n+2)(\gamma_{4n+3} \nu^+_{4n+3})_*
\pi_m(\mathbb{S}^{4n+6})$. In particular, we derive $G_{4n+6}(\mathbb{S}^{4n+3}/H)\supseteq(24,n+2) \gamma_{4n+3\ast}\pi_{4n+6}(\mathbb{S}^{4n+3})\approx\mathbb{Z}_{\frac{24}{(24,n+2)}}$ for $n\geq 2$.
\end{theorem}

Observe that Theorem~\ref{hthm} yields:
\begin{corollary}\label{GHP1}
\begin{enumerate}\thmenumhspace
\item $G_{8n+9}(\mathbb{S}^{8n+3}/H)=0$ and $G_{8n+10}(\mathbb{S}^{84+7}/H)=\gamma_{8n+7\ast}\pi_{8n+10}(\mathbb{S}^{8n+7})$ for $n\not\equiv 0 \pmod 3$.

\item $G_{4n+14}(\mathbb{S}^{4n+3}/H)= \gamma_{4n+3\ast}\pi_{4n+14}(\mathbb{S}^{4n+3})$ for $n\equiv 5,9 \pmod{12}$ with  $n\ge 5$ and
$n\equiv 15,23 \pmod{24}$.
\end{enumerate}
\end{corollary}

Applying \cite[Corollary~2.53]{juno-marek}, we derive:
\begin{corollary} \label{HGP2}
$G_{16n+m}(\mathbb{S}^{16n-1}/H)= \gamma_{16n-1\ast}\pi_{16n+m}(\mathbb{S}^{16n-1})$ for $m=20$, $n\ge 1$ and $m=21$, $ n\ge 2$.
\end{corollary}

Finally, in view of \cite[Proposition~2.54]{juno-marek}, we state:
\begin{proposition}\label{CHP2ex2}
\begin{enumerate}\thmenumhspace
\item  $G_{21}(\mathbb{S}^{11}/H)\supseteq 2\gamma_{11\ast}\pi_{21}(\mathbb{S}^{11})$;

\item  $G_{22}(\mathbb{S}^{11}/H)\supseteq 8\gamma_{11\ast}\pi_{22}(\mathbb{S}^{11})\approx\mathbb{Z}_{63}$;

\item  $G_{22}(\mathbb{S}^{15}/H)\supseteq 4\gamma_{15\ast}\pi_{22}(\mathbb{S}^{15})\approx\mathbb{Z}_{60}$.
\end{enumerate}
\end{proposition}

\section{Miscellanea on lens spaces $L^{2n+1}(l; \mathbf{q})$}\label{sec.4}

Consider the $(2n+1)$-dimensional unit sphere  $\mathbb{S}^{2n+1}$ in the Euclidean $(2n+2)$-space in terms of $(n+1)$-complex coordinates $z = (z_0, z_1,\dots, z_n)$ satisfying $z_0\bar{z}_0 + \dots + z_n\bar{z}_n = 1$.

Let  $l\ge 2$ be a fixed integer, and $q_1,\dots,q_n$ be integers relatively prime to $l$. For
the cyclic group $\mathbb{Z}_l=\big<g\big>$, we define an action  $\mathbb{Z}_l\times \mathbb{S}^{2n+1}\to \mathbb{S}^{2n+1}$ by $(g, (z_0, z_1, \dots, z_n)) \mapsto (e^{2\pi i/l}z_0, e^{2\pi iq_1/l}z_1, \dots, e^{2\pi iq_n/l}z_n)$.
This generates a fixed point free rotation of $\mathbb{S}^{2n+1}$ of order $l$. The orbit space $\mathbb{S}^{2n+1}/\mathbb{Z}_l = L^{2n+1}(l; q_1, \dots, q_n)$ is an orientable $(2n + 1)$-dimensional manifold called a \textit{lens space}.
Thus, a lens space is a special case of a spherical orbit space.

The topological classification of lens spaces is given as follows: spaces
$L^{2n+1}(l; q_1,\dots, q_n)$ and $L^{2n+1}(l; q'_1,\dots, q'_n)$ are homeomorphic if and only
if there is a number $b$ and there are numbers $\varepsilon_j\in\{-1,1\}$ for $j=1,\dots,n$ such that 
$(q_1,\dots,q_n)$ is a permutation of $(\varepsilon_1bq'_1,\dots, \varepsilon_nbq'_n) \pmod l$.

The homotopy classification of lens spaces is given as follows: spaces $L^{2n+1}(l; q_1,\dots,q_n)$ and
$L^{2n+1}(l; q'_1,\dots,q'_n)$ have the same homotopy type if and only if $q_1q_2\cdots q_n \equiv \pm k^nq'_1q'_2\cdots
q'_n \pmod l$ for some integer $0<k<l$. Thus, $L^3(5; 1)$ is not homotopic to $L^3(5; 2)$ while they have the
same homotopy groups. For more information on lens spaces readers are referred to \cite{milnor1} and \cite{olum}.

Write $L^{2n+1}(l; q_1,\dots, q_n) = L^{2n+1}(l; \mathbf{q})$, where $\mathbf{q} = (q_1, \dots, q_n)$.
If $\gamma_{2n+1} \colon  \mathbb{S}^{2n+1}\to  L^{2n+1}(l; \mathbf{q})$ is the quotient map and $g^t\in \mathbb{Z}_l$ then $\gamma_{2n+1} (g^t(z)) =\gamma_{2n+1}(z)$ for $z\in \mathbb{S}^{2n+1}$ and $t = 0, \dots, l-1$. Thus, $\mathbb{Z}_l$ is a group of deck
transformations since $\gamma_{2n+1}$ is a covering map, and we have $\pi_1(L^{2n+1}(l; \mathbf{q}))=
\mathbb{Z}_l$ and $\pi_m(L^{2n+1}(l; \mathbf{q})) = \pi_m(\mathbb{S}^{2n+1})$ for $m>1$.

By Proposition~\ref{l1}, we get $G_{m}(L^{2n+1}(l;\mathbf{q}))\subseteq \gamma_{2n+1*}G_m(\mathbb{S}^{2n+1})$ for $m>1$. By means of Remark~\ref{mukai}, it seems that the equality $G_{m}(L^{2n+1}(l;\mathbf{q}))= \gamma_{2n+1*}G_m(\mathbb{S}^{2n+1})$ for $m>1$  does not hold in general. 

In view of \cite[\textsc{Theorem~A}]{oprea}, it holds $G_{1}(L^{2n+1}(l;\mathbf{q}))= \mathbb{Z}_l$. Because  $G_1(L^{2n+1}(l;\mathbf{q}))\subseteq G_1^f(L^{2n+1}(l;\mathbf{q}))\subseteq \pi_1(L^{2n+1}(l;\mathbf{q}))=\mathbb{Z}_l$, we deduce that  $G_1(L^{2n+1}(l;\mathbf{q}))= G_1^f(L^{2n+1}(l;\mathbf{q}))$  for any map $f\colon A\to L^{2n+1}(l;\mathbf{q})$ satisfying the
conditions as stated in Theorem~\ref{main}.

Further, Pak and Woo have shown in \cite{pak} that
\[G_{2n+1}(L^{2n+1}(l;\mathbf{q})) =\begin{cases}
\mathbb{Z}, & \text{for $n=0,1,3$};\\
2\mathbb{Z}, & \text{for any other $n$.}
\end{cases}\]
Corollary~\ref{CCC} and Remark~\ref{lenss} lead to:
\begin{proposition}\label{lens}
\begin{enumerate}\thmenumhspace
\item $G_1(L^{2n+1}(l;\mathbf{q}))=G_1^{\gamma_{2n+1}}(L^{2n+1}(l;\mathbf{q}))=\mathbb{Z}_l$; 
\item $P_1(L^{2n+1}(l;\mathbf{q}))=P_1^{\gamma_{2n+1}}(L^{2n+1}(l;\mathbf{q}))=\mathbb{Z}_l$;

\item $G_m^{\gamma_{2n+1}}(L^{2n+1}(l;\mathbf{q}))=P_m^{\gamma_{2n+1}}(L^{2n+1}(l;\mathbf{q}))=\operatorname{Ker}[\gamma_{2n+1},-]$ for $m\geq 1$;
\item  $P_m^{\gamma_{2n+1}}(L^{2n+1}(l;\mathbf{q}))=P_m(L^{2n+1}(l;\mathbf{q}))=\gamma_{{2n+1}\ast}G_m(\mathbb{S}^{2n+1})$ for $m>1$;
\item $G_m^{\gamma_3}(L^3(l;\mathbf{q}))=P_m^{\gamma_3}(L^{3}(l;\mathbf{q}))=G_m(L^3(l;\mathbf{q}))=P_m(L^{3}(l;\mathbf{q}))=\gamma_{3\ast}\pi_m(\mathbb{S}^3)$ for $m\geq 1$.
\end{enumerate}
\end{proposition}

Notice that $L^{2n+1}(2;\mathbf{q})=\mathbb{R}P^{2n+1}$ and  $L^{2n+1}(l; 1,\dots, 1) = \mathbb{S}^{2n+1}/\mathbb{Z}_l$, where the action $\mathbb{Z}_l\times \mathbb{S}^{2n+1}\to \mathbb{S}^{2n+1}$ is determined by the canonical inclusion $\mathbb{Z}_l\hookrightarrow \mathbb{S}^1$. 

By the above, there is a homotopy equivalence  $L^{2n+1}(l; q_1,\dots, q_n) = L^{2n+1}(l; q_1\cdots q_n,1,\linebreak \dots,1)$ for $l> 2$.
Hence, an existence of a solution of the equation $q_1\cdots q_n \equiv \pm k^n \pmod l$ implies the homotopy equivalence $L^{2n+1}(l; q_1,\dots, q_n) \simeq \mathbb{S}^{2n+1}/\mathbb{Z}_l$ for $l> 2$.

The groups $G_m(\mathbb{R}P^{2n+1})$ and  $G_m(\mathbb{S}^{2n+1}/\mathbb{Z}_l)$ for $l>2$ have been studied in Section~\ref{sec.2}.

\subsubsection*{Acknowledgments}We would like to thank  D.L.\ Gon\c calves for posing a possible generalization of the Oprea's result \cite{oprea}. The first author is grateful to Department of Mathematics of IGCE--Unesp (Brazil) for its kind hospitality, where the paper has been completed.

\end{document}